\documentclass[12pt]{amsart}
\usepackage{amsmath,amssymb,enumerate,mathtools}
\newtheorem{theorem}{Theorem}[section]
\newtheorem{claim}{Claim}[theorem]
\newtheorem{lemma}[theorem]{Lemma}

\newtheorem{corollary}[theorem]{Corollary}
\newtheorem{conjecture}[theorem]{Conjecture}
\theoremstyle{definition}

\newcommand{\bR}{\mathbb R}

\newcommand{\bZ}{\mathbb Z}

\newcommand{\bQ}{\mathbb Q}

\newcommand{\cW}{\mathcal{W}}

\newcommand{\cX}{\mathcal{X}}

\DeclareMathOperator{\cl}{cl}

\newcommand{\del}{\!\setminus\!}
\newcommand{\con}{/}

\newcommand{\pr}{\mathbf{Pr}}
\newcommand{\ex}{\mathbf{E}}
\newcommand{\var}{\mathbf{Var}}

\title{Doubly exponentially many Ingleton matroids}
\author{Peter Nelson}
\address{Department of Combinatorics and Optimization, University of Waterloo, Waterloo, Canada}
\author{Jorn van der Pol}
\address{Department of Mathematics and Computer Science, Eindhoven University of Technology, Eindhoven, The Netherlands}
\thanks{The first author was supported by a grant from the Natural Sciences and Engineering Research Council of Canada}
\thanks{The second author was supported by the Netherlands Organisation for Scientific Research (NWO) grant~613.001.211.}

\begin{document}

\begin{abstract}A matroid is \emph{Ingleton} if all quadruples of subsets of its ground set satisfy Ingleton's inequality. In particular, representable matroids are Ingleton. We show that the number of Ingleton matroids on ground set $[n]$ is doubly exponential in $n$; it follows that almost all Ingleton matroids are non-representable.
\end{abstract}

\maketitle

\section{Introduction}

Ingleton's condition~\cite{Ingleton1971} is a well-known linear inequality that holds for representable matroids but not matroids in general; it states that for all $A,B,C,D \subseteq E$ in a representable matroid $M = (E,r)$, 
\begin{multline*}
r(A \cup B) + r(A \cup C) + r(A \cup D) + r(B \cup C) + r(B \cup D) \\
\ge r(A) + r(B) + r(A \cup B \cup C) + r(A \cup B \cup D) + r(C \cup D).
\end{multline*}

An arbitrary matroid is \emph{Ingleton} if the above inequality is satisfied for all choices of $A,B,C,D$.  The class of Ingleton matroids is closed under minors and duality (see, for example, Lemmas~3.9 and~4.5 in~\cite{Cameron14}) and clearly all representable matroids are Ingleton. A natural question is to what extent a converse of this last statement holds: that is, do Ingleton matroids tend to be representable? We prove here that they do not. For $n \ge 12$, the number of representable matroids on $[n]$ is at most $2^{0.25n^3}$~\cite{Nelson2016}; our main result is the following. 
\begin{theorem}\label{main}
	For all sufficiently large $n$ and all $0 < r < n$, the number of Ingleton matroids with ground set $[n]$ is at least $2^{0.486 \tfrac{\log (r(n-r))}{r(n-r)}\binom{n}{r}}$. 
\end{theorem}
Even when $r = 4$, this eclipses the upper bound on the number of representable matroids on $[n]$ with no restriction on rank; thus, almost all Ingleton matroids are non-representable. When $r = \lfloor n/2 \rfloor$, the above formula is around $2^{\frac{1.94}{n^2}\binom{n}{n/2}}$ which is doubly exponential in $n$, and even somewhat resembles the number of all matroids on $[n]$, which is $2^{\Theta\left(\frac{1}{n}\binom{n}{\lfloor n/2 \rfloor}\right)}$  with a constant between $1$ and $2+o(1)$~\cite{BansalPendavinghvanderPol2015}. We conjecture, however, that general matroids tend not to be Ingleton:

\begin{conjecture}
	There is a constant $c$ such that the number of Ingleton matroids on $[n]$ is $2^{(c+o(1)) \tfrac{\log n}{n^2}\binom{n}{\lfloor n/2 \rfloor}}$.  
\end{conjecture}  
Conceivably, the constant $c$ could be equal to $2$, or even be the one of around $1.94$ obtained by our proof.

In what follows, we assume some familiarity with matroid theory; see~\cite{Oxley2011}. In particular, a \emph{nonbasis} of a rank-$r$ matroid $M$ is a set in $\binom{E(M)}{r}$ that is not a basis of $M$; and a matroid is \emph{paving} if all its nonbases are circuits and \emph{sparse paving} if $M$ and $M^*$ are both paving. Logarithms are all base-two.

As part of the proof of Theorem~\ref{main}, we also characterize exactly which sparse paving matroids are Ingleton, and as a result easily derive the following theorem:
\begin{theorem}\label{forty}
	There are exactly $41$ excluded minors for the class of Ingleton sparse paving matroids: the matroids $U_{0,2}\oplus U_{1,1}$ and $U_{2,2} \oplus U_{0,1}$, and the $39$ rank-$4$ non-Ingleton sparse paving matroids on eight elements. 
\end{theorem}

The fact that this set is even finite is curious; the class of all Ingleton matroids, by contrast, has an infinite set of excluded minors, constructed in~\cite{MayhewNewmanWhittle2009}. In fact, their techniques show that every $\bR$-representable matroid is a minor of an excluded minor for the class of Ingleton matroids. 

Theorems~\ref{main} and~\ref{forty} together imply that the Ingleton matroids are a `large' minor-closed class of matroids (in the sense of numbering at least $2^{2^n/\mathrm{poly}(n)}$) that omits $39$ different sparse paving matroids.
It was conjectured in~\cite{MayhewNewmanWelshWhittle2011} that any minor-closed class not containing all sparse paving matroids is asymptotically vanishing; our result shows that such a class may still be `large'.


\section{Representing matroids with few nonbases}

\begin{lemma}\label{hall}
	Let $M$ be a rank-$r$ matroid in which each set $\cW$ of nonbases with $|\cW| > 1$ satisfies $|\cap \cW| \le r - |\cW|$. Then $M$ is $\bR$-representable. 
\end{lemma}
\begin{proof}
	Let $\cX = \{X_1, \dotsc, X_t\}$ be the set of nonbases of $M$; note that $0 \le |\cap \cX| \le r-t$, so $t \le r$. Let $A$ be an $[r] \times E$ real matrix so that the nonzero entries of $A$ are algebraically independent over $\bQ$, and $A_{i,e} = 0$ if and only if $i \in [t]$ and $e \in X_i$. We prove that $M = M(A)$. It is clear that for every nonbasis $W$ of $M$, the matrix $A[W]$ has a zero row so is singular; it remains to show that $A[B]$ is nonsingular for each basis $B$ of $M$.
	
	Let $B$ be a basis of $M$. Consider the bipartite graph $G$ with bipartition $([r],B)$ for which $(i,e)$ is an edge if and only if $A_{i,e} \ne 0$. Note that each $i \in \{t+1, \dotsc, r\}$ has degree $r$ in $G$. For each $S \subseteq [r]$, let $N(S)$ denote the set of vertices in $B$ that are adjacent to a vertex in $S$. We argue that $|N(S)| \ge |S|$ for each $S \subseteq [r]$; it will follow from Hall's theorem that $G$ has a perfect matching. 
		
	Let $S \subseteq [r]$. If $S \not\subseteq [t]$ then clearly $N(S) = B$ and so $|N(S)| = r \ge |S|$. If $S \subseteq [t]$, then by hypothesis the set $\bigcap_{s \in S}X_s$ has size at most $r-|S|$, so $B$ contains at least $|S|$ elements $e$ for which there is some $s \in S$ with $e \notin X_s$. Each such $e$ is adjacent to $s$, so $|N(S)| \ge |S|$ as required. Therefore $G$ has a perfect matching. 
			
	Let $B = \{b_1, \dotsc, b_r\}$ and $S_r$ denote the symmetric group on $[r]$; note that $A[B]$ is singular if and only if the determinant $\sum_{\sigma \in S_r}\prod_{i \in [r]}A_{i,b_{\sigma(i)}}$ is zero. This determinant is a polynomial in the entries of $A$, with integer coefficients, whose nonzero monomials are algebraically independent over $\bQ$, and since $G$ has a perfect matching, some monomial is nonzero. It follows that the determinant is nonzero, so $A[B]$ is nonsingular, as required.
\end{proof}

\begin{lemma}\label{atmostfour}
	Every matroid with at most four nonbases is $\bR$-representable.
\end{lemma}
\begin{proof}
	Let $M$ be a minor-minimal counterexample. Note that $M$ is simple, that $r(M) \ge 3$ , and that $M^*$ is also a minor-minimal counterexample, so $M$ is cosimple with $r^*(M) \ge 3$. If $M$ has an element $e$ in no nonbases, then since $e$ is not a coloop, $M$ is the free extension of $M \del e$ by $e$; since $M \del e$ is $\bR$-representable, so is $M$, a contradiction. So every element is in a nonbasis of $M$. Dually, no element is in all nonbases of $M$. 
	
	If $M$ has a dependent set $Y$ of size $r(M)-1$, then $Y \cup \{e\}$ is a nonbasis for each $e \in E(M)-Y$. Since no element of $Y$ is in four nonbases, this gives $|E(M)-Y| \le 3$, so $|E(M)|\le r+2$, giving $r^*(M) \le 2$, a contradiction. Therefore every circuit of $M$ is spanning, so $M$ is a paving matroid; dually, $M$ is a sparse paving matroid. 
	
	If $e \in E(M)$ is in exactly one nonbasis $X$, then $M$ is the principal extension of the flat $X - \{e\}$ in $M \del e$, so $M$ is $\bR$-representable, a contradiction. Therefore every $e \in E(M)$ is in at least two nonbases. Dually, every element lies outside at least two nonbases. Therefore $M$ has exactly four nonbases, and every element is in exactly two of them.
	  
	If $r(M) = 3$, then $M$ has four triangles, so there are $12$ pairs $(e,T)$ where $T$ is a triangle containing $e$, and every element is in two triangles, so there are also $2|E(M)|$ such pairs $(e,T)$. Thus $|E(M)| = 6$ and so $M$ is $\bR$-representable, a contradiction. 
	
	Suppose, therefore, that $r(M) \ge 4$. By Lemma~\ref{hall}, we may assume that there is some set $\cX$ of nonbases of $M$ with $|\cX| > 1$ such that $|\cap \cX| > r - |\cX|$. Since no element is in three nonbases, if $|\cX| > 2$ then $|\cap \cX| = 0 \le r - |\cX|$, so we must have $|\cX| = 2$ and thus there are nonbases $X_1,X_2$ with $|X_1 \cap X_2| = r-1$. This contradicts the fact that $M$ is a sparse paving matroid. 
\end{proof}

\section{Ingleton Matroids}

In this section, we use the well-known fact that $H \subseteq \binom{E}{r}$ is the set of nonbases of a sparse paving matroid on $E$ if and only if no two elements of $H$ have intersection of size exactly $r-1$.

\begin{lemma}\label{ingletonviolations}
	Let $M$ be a rank-$r$ sparse paving matroid. Sets $A,B,C,D$ violate the Ingleton inequality in $M$ if and only if there are pairwise disjoint sets $X_1,X_2,X_3,X_4,Y,Z_1,Z_2 \subseteq E(M)$ such that 
	\begin{itemize}
		\item $|X_i| = 2$ for each $i \in [4]$ while $|Y \cup Z_1 \cup Z_2| = r-4$, 
		\item $A = X_1 \cup Y \cup Z_1 \cup Z_2$,
		\item $B = X_2 \cup Y \cup Z_1 \cup Z_2$, 
		\item $C = X_3 \cup Y \cup Z_1$, and 
		\item $D = X_4 \cup Y \cup Z_2$,
	\end{itemize}
	while each of $A \cup B, A\cup C,A \cup D,B \cup C$ and $B \cup D$ is a circuit-hyperplane of $M$, and $C \cup D$ is a basis. 
\end{lemma}
\begin{proof}
	If the above conditions are satisfied, then the Ingleton inequality is evidently violated. Conversely, let $A,B,C,D$ violate the Ingleton inequality. For each matroid $N$ with $A \cup B \cup C \cup D \subseteq E(N)$, let 
	\[h_1(N) = r_N(A) + r_N(B) + r_N(A \cup B \cup D) + r_N(A \cup C \cup D) + r_N(C \cup D)\] and \[h_2(N) = r_N(A \cup B) + r_N(A \cup C) + r_N(A \cup D) + r_N(B \cup C) + r_N(B \cup D),\] so $h_1(M) > h_2(M)$ by assumption. 
	\begin{claim}
		$A \cup B$, $A \cup C$, $A \cup D$, $B \cup C$ and $B \cup D$ are circuit-hyperplanes.
	\end{claim} 
	\begin{proof}[Proof of claim:]
		Suppose otherwise. Let $M'$ be obtained from $M$ by relaxing each circuit-hyperplane other than those among the five sets above, so $M'$ is sparse paving and has at most four circuit-hyperplanes. By Lemma~\ref{atmostfour}, $M'$ is $\bR$-representable, so $h_1(M') \le h_2(M')$. By construction, we have $h_2(M') = h_2(M)$, and since $M'$ is freer than $M$, we have $h_1(M') \ge h_1(M)$. Therefore \[h_2(M) < h_1(M) \le h_1(M') \le h_2(M') = h_2(M),\] a contradiction.
	\end{proof}
	\begin{claim}
		$|A| = |B| = r-2$, and the sets $A \cup B \cup C$, $A \cup B \cup D$ and $C \cup D$ are spanning in $M$. 
	\end{claim}
	\begin{proof}[Proof of claim:]
		The first claim gives $h_2(M) = 5r-5$, so by assumption $h_1(M) \ge 5r-4$. If $A \cup B \subseteq \cl_M(A)$, then we have \[h_2(M) = r_M(A) + r_M(A \cup B \cup C) + r_M(A \cup B \cup D) + r_M(B \cup C) + r_M(B \cup D)\] and $h_2(M) - h_1(M) = r_M(B \cup C) + r_M(B \cup D) - r_M(B) - r_M(C \cup D) \ge 0$ by submodularity, a contradiction. So $A \cup B \not\subseteq \cl_M(A)$; since $A \cup B$ is a circuit, it follows that $|A| = r_M(A) \le r-2$ and, symmetrically, that $|B| = r_M(B) \le r-2$. Therefore
		\begin{align*} 
			3r &\ge r_M(A \cup B \cup C) + r_M(A \cup B \cup D) + r_M(C \cup D) \\
			   &= h_1(M)-r_M(A)-r_M(B)\\
			   &\ge (5r-4)-2(r-2) = 3r,
		\end{align*}
		so we have equality throughout, and $r_M(A \cup B \cup C) = r_M(A \cup B \cup D) = r_M(C \cup D) = r$ while $|A| = |B| = r-2$. 
	\end{proof}

	For each nonempty subset $S$ of $\{A,B,C,D\}$, write $J_S$ for the collection of elements belonging to all sets in $S$ but no sets in $\{A,B,C,D\}-S$, and let $n_S = |J_S|$. For example, $n_{AB}$ denotes $|(A \cap B)-(C \cup D)|$ (we omit commas and braces). 
	
	Since $A \cup C$ and $B \cup C$ are circuit-hyperplanes in a sparse paving matroid, 
	\begin{align*}
		2 &\le |(A \cup C) - (B \cup C)| \\
		& = |A - (B \cup C)|\\
		& = |A-B| - |(A \cap C)-B| \\
		& = 2 - |(A \cap C)-B|,
	\end{align*}
	so $(A \cap C)-B = \varnothing$, giving $n_{AC} = n_{ACD} = 0$. Using the symmetry between $A$ and $B$ and between $C$ and $D$, we also have $n_{AD} = n_{BC} = n_{BD} = n_{BCD} = 0$. Therefore $n_C + n_{CD} = n_C + n_{CD} + n_{BCD} + n_{BD} = |C-A| = |C \cup A| - |A| = 2$.
	
	Since $A \cup C$ and $A \cup D$ are circuit-hyperplanes, we have
	$2 \le |(A \cup C) - (A \cup D)| = n_D + n_{BD} = n_D = 2-n_{CD}$, from which we get $n_{CD} = 0$ and $n_D = 2$, and symmetrically $n_C = 2$. Moreover $n_A = n_A + n_{AC} + n_{AD} + n_{ACD} = |A-B| = |A \cup B| - |B| = r-(r-2)= 2$, and symmetrically $n_B = 2$.
	
	The four undetermined $n_S$ thus far are $n_{AB},n_{ABCD},n_{ABC}$ and $n_{ABD}$; all others have been shown to be zero except $n_A = n_B = n_C = n_D = 2$. Using the fact that $C \cup D$ is spanning, we thus have 
	\[r \le |C \cup D| = n_{ABCD} + n_{ABC} + n_{ABD} + n_C + n_D.\]
	On the other hand, 
	\[r-2 = |A| = n_{ABCD} + n_{ABC} + n_{ABD} + n_{AB} + n_A;\]
	since $n_A = n_C = n_D = 2$, these together imply that $n_{AB} = 0$. The above also gives $n_{ABCD} + n_{ABC} + n_{ABD} = r-4$. Now setting $(X_1,X_2,X_3,X_4,Y,Z_1,Z_2) = (J_A,J_B,J_C,J_D,J_{ABCD},J_{ABC},J_{ABD})$ gives the required structure. Finally, we see that $|C \cup D| = n_{ABCD} + n_{ABD} + n_{ABC} + n_C + n_D = (r-4)+4 = r$; since $C \cup D$ is spanning, it must be a basis. 
\end{proof}

A simpler characterisation of these matroids below follows with $K = Z \cup Y_1 \cup Y_2$ and the $P_i$ equal to some ordering of the $X_i$ above.
\begin{corollary}\label{symmdiff}
	Let $M$ be a sparse paving matroid. Then $M$ is non-Ingleton if and only if there are pairwise disjoint sets $P_1,P_2,P_3,P_4,K$ so that  $|K| = r-4$ and $|P_i| = 2$ for each $i$, and exactly five of the six sets of the form $K \cup P_i \cup P_j \colon i \ne j$ are circuit-hyperplanes of $M$.
\end{corollary}

	As observed in \cite{Cameron14}, If $M$ is a matroid for which the above condition holds, then it also holds in the eight-element, rank-$4$ matroid $(M \con K)|(P_1 \cup P_2 \cup P_3 \cup P_4)$; therefore, every non-Ingleton sparse paving matroid has an eight-element, rank-$4$ non-Ingleton sparse paving matroid as a minor. Mayhew and Royle~\cite{MayhewRoyle2008} showed that there are precisely $39$ such matroids; for every such matroid $N$, the V\'{amos} matroid $V_8$ can be obtained from $N$ by a sequence of circuit-hyperplane relaxations. (We remark that~\cite{MayhewRoyle2008} uses different terminology from ours, calling these matroids `Ingleton non-representable' rather than `non-Ingleton'.)
	The unique minor-minimal matroids that are not sparse paving are $U_{0,2} \oplus U_{1,1}$ and $U_{2,2} \oplus U_{0,1}$; together these facts imply Theorem~\ref{forty}.

\section{Counting Ingleton Matroids}

The proof of the following theorem uses techniques from~\cite[Proposition~2.1]{CooperMubayi2014}. 

\begin{theorem}
	There exists $n_0$ such that for all $n \ge n_0$ and all $0 < r < n$, the number of rank-$r$ Ingleton matroids with ground set $[n]$ is at least $2^{0.486 \frac{\log (r(n-r))}{r(n-r)} \binom{n}{r}}$.
\end{theorem}

\begin{proof}
	We may assume that $2 \le r \le \frac{n}{2}$, since otherwise the theorem is trivial or follows by duality.  Write $N = \binom{n}{r}$ and $d = r(n-r)$ for the number of vertices and the valency of the Johnson graph $J(n,r)$. 
	
	For $x \in \bR$, let $f(x) = 1 - \tfrac{1}{2}x - \tfrac{1}{64}x^4$. We show that if $c > 0$ is a real number and $\gamma < cf(c)$, then there are at least $ 2^{\gamma\frac{\log d}{d}N}$ Ingleton sparse paving matroids of rank $r$ on $[n]$, provided $n$ is sufficiently large; the result as stated follows with $c = 0.95$ and $\gamma = 0.486$. 
	
	Given $c$ and $\gamma$, let $\alpha$ be such that $\gamma/c < \alpha < f(c)$ and let $\epsilon = f(c) - \alpha$, so $1-f(c) + \epsilon = 1-\alpha$.
	Set $k = \left\lfloor c \frac{N}{d}\right\rfloor$, so $(1-o(1))\tfrac{c}{d} \le \tfrac{k}{N} \le \tfrac{c}{d}$. Pick a $k$-set $H$ of vertices in $J(n,r)$ uniformly at random from among all $k$-subsets of vertices and write $E(H)$ for the set of unordered pairs of vertices in $H$ joined by an edge in $J(n,r)$, and $e_H$ for $|E(H)|$.
	\begin{claim}\label{eestimates}
		$\ex(e_H) \le \frac{ck}{2}$ and $\var(e_H) = o(k^2)$. 
	\end{claim}
	\begin{proof}[Proof of claim:]
		We have 
		\[
		\ex(e_H) = \frac{1}{2} d N \frac{\binom{N-2}{k-2}}{\binom{N}{k}} \le \frac{dk^2}{2N} \le \frac{ck}{2}.
		\]
		Let $\Theta$ denote the set of ordered pairs $(e,f)$ of edges of $J(n,r)$. Write $\Theta_j$, $j\in\{0,1,2\}$, for the set of pairs in $\Theta$ that span $4-j$ vertices. Note that $|\Theta| = \frac{1}{4}d^2 N^2$, while $|\Theta_1| = Nd(d-1) \le Nd^2$ and $|\Theta_2| = \frac{1}{2}dN$. Now, using the fact that $\binom{N-\ell}{k-\ell}/\binom{N}{k} = (1+o(1))(k/N)^{\ell}$ for each constant $\ell$, we have
		\begin{align*}
			\var(e_H) &= \sum_{\substack{(e,f) \in \Theta}}\big[\pr(e,f \in E(H)) -\pr(e \in E(H))\pr (f \in E(H))\big]\\
			&= |\Theta_0| \left(\frac{\binom{N-4}{k-4}}{\binom{N}{k}}-\frac{\binom{N-2}{k-2}^2}{\binom{N}{k}^2}\right) + |\Theta_1| \left(\frac{\binom{N-3}{k-3}}{\binom{N}{k}}-\frac{\binom{N-2}{k-2}^2}{\binom{N}{k}^2}\right) \\
			&\qquad\qquad + |\Theta_2| \left(\frac{\binom{N-2}{k-2}}{\binom{N}{k}}-\frac{\binom{N-2}{k-2}^2}{\binom{N}{k}^2}\right) \\
			&\le \frac{1}{4}N^2d^2 o\left(\frac{k^4}{N^4}\right) + Nd^2\left(\frac{k}{N}\right)^3 + \frac{1}{2} d N \left(\frac{k}{N}\right)^2 \\
			&= o(d^{-2}N^2) + O(d^{-1}N) + O(d^{-1} N) = o(k^2),
		\end{align*}
		since $k = (1+o(1))cN/d$. 
	\end{proof}
	
	Let $\Omega$ be the set of all pairs $(\{P_1,P_2,P_3,P_4\},K)$ where $P_1,P_2,P_3,P_4,K$ are pairwise disjoint subsets of $[n]$ with $|P_i|= 2$ and $|K| = r-4$ (note that the collection of $P_i$ is unordered). Now
	\begin{align*}|\Omega| &= \frac{1}{4!}\binom{8}{2,2,2,2}\binom{n}{8} \binom{n-8}{r-4}\\
	&= \frac{8!}{2^4 \cdot 4!}\binom{n}{r}\frac{r!(n-r)!}{8!(r-4)!(n-r-4)!}\\
	&\le \frac{d^4}{2^7 \cdot 3}N.
	\end{align*}
	
	For each $\omega \in \Omega$, let $U(\omega) = \{K \cup P_i \cup P_j\colon \{i,j\} \in \binom{[4]}{2}\}$, so $|U(\omega)| = 6$. For each $H$ and each $i \in \{0,\dotsc,6\}$, let $b_{i,H}$ denote the number of $\omega$ in $\Omega$ for which $|H \cap U(\omega)| = i$.
	\begin{claim}\label{bestimates}
		$\ex(b_{5,H}) \le \frac{c^4k}{64}$ and $\ex(b_{6,H}) = o(k)$ while $\var(b_{5,H}) = o(k^2)$. 
	\end{claim}
	\begin{proof}[Proof of claim:]
	 The claim is trivial when $r < 4$ since $\Omega$ is empty, so suppose that $r \ge 4$. Given $\omega \in \Omega$, the probability that $|H \cap U(\omega)| = i$ is $\binom{6}{i}\binom{N-6}{k-i}/\binom{N}{k} \le \binom{6}{i}\left(k/N\right)^i \le \binom{6}{i}c^id^{-i}$, so
	\[\ex(b_{i,H}) \le |\Omega|\binom{6}{i}c^id^{-i} \le \binom{6}{i}\frac{c^{i}d^{4-i}}{2^7 \cdot 3}N \le \binom{6}{i}\frac{c^{i-1}d^{5-i}k}{2^7 \cdot 3},\]
	giving $\ex(b_{5,H}) \le \frac{c^4k}{64}$ and $\ex(b_{6,H}) = o(k)$ as required. 
	
	
	Let $\Pi = \Omega^2$, so $|\Pi| = |\Omega|^2 \le d^8 N^2$. Let $\Pi_0 := \{(\omega, \omega) : \omega \in \Omega\} \subseteq \Pi$, let $\Pi_2$ be the set of all $(\omega_1, \omega_2) \in \Pi$ for which $U(\omega_1) \cap U(\omega_2) = \emptyset$, and let $\Pi_1 = \Pi \setminus (\Pi_0 \cup \Pi_2)$. Since $U(\omega)$ contains $6$ vertices of $J(n,r)$ for each $\omega \in \Omega$, symmetry and a counting argument gives that for each vertex $v$ of $J(n,r)$, we have 
	\[|\{\omega \in \Omega\colon v \in U(\omega)\}| = \frac{6|\Omega|}{N} = O(d^4).\]
	It follows that $|\Pi_1| = O(d^8 N)$. Call an $\omega \in \Omega$ \emph{bad} for  $H$ if $|H \cap U(\omega)| = 5$. Recall that the probability that a given $\omega$ is bad is $6\binom{N-6}{k-5}/\binom{N}{k} = (6+o(1))k^5/N^5$. Note that $\omega$ is determined uniquely by any four sets in $U(\omega)$; it follows that if $(\omega_1,\omega_2) \in \Pi_1$ then $U(\omega_1)$ and $U(\omega_2)$ have at most three sets in common, so if both $\omega_1$ and $\omega_2$ are bad, then $H$ contains at least seven of the sets in $U(\omega_1)$ and $U(\omega_2)$. Since $|U(\omega_1) \cup U(\omega_2)| \le 10$, a pair $(\omega_1,\omega_2) \in \Pi_1$ is thus bad with probability at most $\binom{10}{7}\binom{N-7}{k-7}/\binom{N}{k} \le \binom{10}{7}(k/N)^7 = O(d^{-7})$. If $(\omega_1,\omega_2) \in \Pi_0$ then $\omega_1$ and $\omega_2$ are both bad with probability $(6+o(1))k^5/N^5 = O(d^{-5})$. If $(\omega_1,\omega_2) \in \Pi_2$ then both are bad with probability $36\binom{N-12}{k-10}/\binom{N}{k} = (36+o(1))k^{10}/N^{10}$. Therefore
	\begin{align*}
		\var(b_{5,H}) &= \sum_{(\omega_1,\omega_2) \in \Pi}\big[\pr(\omega_1,\omega_2 \text{ bad})-\pr(\omega_1 \text{ bad})\pr(\omega_2 \text{ bad})\big]\\
		&\le |\Pi_2|\left(\frac{(36+o(1))k^{10}}{N^{10}} - \left(\frac{(6 + o(1))k^5}{N^5}\right)^2\right) + |\Pi_0|O(d^{-5}) \\
		&\qquad\qquad + |\Pi_1| O(d^{-7}) \\
		&\le d^8N^2 o(k^{10}/N^{10}) + O(d^{-1}N) + O(d N) \\
		&= o(d^{-2}N^2) + O(dN).
	\end{align*}
	Now $d^{-2}N^2 = (1+o(1))k^2$, and, using $4 \le r \le \tfrac{n}{2}$, we have $dN = r(n-r)\binom{N}{r} = o\left(\frac{1}{r^2(n-r)^2}\binom{N}{r}^2\right) = o(k^2)$. It follows that $\var(b_{5,H}) = o(k^2)$ as required.  
	\end{proof}
	By the two claims, the random variables $e_H$ and $b_{5,H}$ have means at most $\frac{ck}{2}$ and $\frac{c^4k}{64}$ respectively, and both have standard deviations in $o(k)$; it follows by Chebyshev's inequality that $\pr(e_H > (\frac{c}{2} + \frac{\epsilon}{3})k) = o(1)$ and $\pr(b_{5,H} > (\frac{c^4}{64} + \frac{\epsilon}{3})k) = o(1)$. Since $\ex(b_{6,H}) \in o(k)$, Markov's inequality gives $\pr(b_{6,H} > \frac{\epsilon}{6}k) = o(1)$. Therefore, with probability $1-o(1)$, we have 
	\[e_H + b_{5,H} + 2b_{6,H} \le (\tfrac{c}{2} + \tfrac{c^4}{64} + \epsilon)k = (1-f(c) + \epsilon)k = (1-\alpha) k.\] 
	
	Call a set $W \subseteq \binom{[n]}{r}$ \emph{good} if $e_W = b_{5,W} = b_{6,W} = 0$. Each set $H \subseteq \binom{[n]}{r}$ of size $k$ contains a good set $W$ of size $|H|- e_H - b_{5,H} - 2b_{6,H}$. With probability $1-o(1)$ we have $e_H + b_{5,H} + 2b_{6,H} \le (1-\alpha)k$ and so $|W| \ge k - (1-\alpha)k = \alpha k$; thus there are at least $(1-o(1))\binom{N}{k}$ different choices of $H$ that contain a good set $W$ of size at least $\alpha k$. On the other hand, each good set $W$ of size at least $\alpha k$ is contained in at most $\binom{N}{(1-\alpha)k}$ different $H$; therefore the number of good sets is at least $\nu = (1-o(1))\tbinom{N}{k}/\tbinom{N}{(1-\alpha)k}$. We have 
	\begin{align*}
		\log \nu &= \log \tbinom{N}{k} - \log \tbinom{N}{(1-\alpha)k} - o(1)\\
		&\ge k \log(N/k) - k \log(eN/(1-\alpha)k) - o(1)\\
		&=(1-o(1))\alpha k \log(N/k)\\
		&=(1-o(1))\alpha c \frac{\log d}{d}N.
	\end{align*}
	But for large $n$ we have $(1-o(1))\alpha c > \gamma$, so there are at least $2^{\gamma \frac{\log d}{d}N}$ good stable sets.  By Corollary~\ref{symmdiff}, each such set is the collection of circuit-hyperplanes of an Ingleton sparse paving matroid of rank $r$ on ground set $[n]$; the theorem follows. 
\end{proof}

We have attempted to optimize the constant $0.486...$ in the exponent as much as possible within the constraints of our techniques; the proof can be simplified to use  first rather than second moments, at the expense of a lowered constant of around $0.4$. One case where the constant can certainly be improved is where the rank $r$ (or, dually, the corank $n-r$) is constant, in which case the estimate on $|\Omega|$ can be improved by an asymptotically significant factor of $(1-\tfrac{1}{r})(1 - \tfrac{2}{r})(1 - \tfrac{3}{r})$. Carrying through this better estimate has the effect of slightly increasing the constant in the exponent further towards $0.5$, giving $0.5$ exactly when $r \le 3$, and roughly $0.498$ when $r = 4$.

We complement the above enumeration result, which is based on the construction of a large family of sparse paving matroids each of which contain roughly $\frac{1}{r(n-r)}\binom{n}{r}$ circuit-hyperplanes, by a construction that shows that sparse paving Ingleton matroids with many more circuit-hyperplanes exist.

We sketch the construction, which is originally due to Graham and Sloane~\cite{GrahamSloane1980}. Suppose that $0 < r < n$. The function $c\colon V(J(n,r)) \to \bZ_n$ given by $c(X) = \sum_{x \in X} x \mod n$ is a proper vertex colouring of $J(n,r)$. It follows that for each $\gamma \in \bZ_n$, the set $c^{-1}(\gamma) \equiv \{X : c(X) = \gamma\}$ is a stable set in $J(n,r)$ and hence $S(n,r,\gamma) := \left([n], \binom{[n]}{r}\setminus c^{-1}(\gamma)\right)$ is the sparse paving matroid whose set of circuit-hyperplanes is $c^{-1}(\gamma)$.

\begin{lemma}\label{lemma:S-ingleton}
	$S(n,r,\gamma)$ is Ingleton.
\end{lemma}
\begin{proof}
	For the sake of contradiction, suppose that $A, B, C, D \subseteq [n]$ violate Ingleton's inequality and obtain $K, P_1, P_2, P_3, P_4$ as in Corollary~\ref{symmdiff}. Write $P_i = \{p_i, p'_i\}$. We may assume that $K \cup P_i \cup P_j$ is a circuit-hyperplane for all $\{i,j\} \in \binom{[4]}{2}\setminus\{\{3,4\}\}$, while $K \cup P_3 \cup P_4$ is a basis. Define $\gamma' = \gamma - \sum_{x \in K} x \mod n$. It follows that
	\[
		p_1 + p'_1 + p_2 + p'_2 = p_1 + p'_1 + p_3 + p'_3 = p_2 + p'_2 + p_4 + p'_4 = \gamma',
	\]
	so in particular
	\[
		p_3 + p'_3 + p_4 + p'_4 = \gamma',
	\]
	which implies that $c(K\cup P_3 \cup P_4) = \gamma$, contradicting that $K \cup P_3 \cup P_4$ is a basis of $S(n,r,\gamma)$.
\end{proof}

\begin{corollary}
	For all $0 < r < n$, there exists a sparse paving Ingleton matroid of rank $r$ on ground set $[n]$ that has at least $\frac{1}{n}\binom{n}{r}$ circuit-hyperplanes.
\end{corollary}

\begin{proof}
	As $\{c^{-1}(\gamma) : \gamma \in \bZ_n\}$ partitions $V(J(n,r))$, there is $\gamma_0 \in \bZ_n$ such that $|c^{-1}(\gamma_0)| \ge \frac{1}{n} \binom{n}{r}$. Consequently, the matroid $S(n,r,\gamma_0)$, which is Ingleton by Lemma~\ref{lemma:S-ingleton}, has at least $\frac{1}{n}\binom{n}{r}$ circuit-hyperplanes.
\end{proof}



	
	
	
	


\bibliographystyle{alpha}
\bibliography{bib}

\end{document}